\documentclass[letterpaper,10pt,conference,romanappendices]{ieeeconf}
\IEEEoverridecommandlockouts
\usepackage{amsmath,amssymb,theorem}
\usepackage{graphicx}
\usepackage{algorithm,algorithmic}
\usepackage{psfrag}
\usepackage{resizegather}

\usepackage{color,xspace}
\usepackage{subfigure}
\usepackage{tikz}
\usepackage{url}
\usetikzlibrary{automata,positioning}

\newtheorem{theorem}{Theorem}[section]
\newtheorem{lemma}[theorem]{Lemma}

\newtheorem{remark}[theorem]{Remark}

\newtheorem{proposition}[theorem]{Proposition}

\newcommand{\real}{{\mathbb{R}}}
\newcommand{\realpositive}{\mathbb{R}_{>0}}
\newcommand{\realnonnegative}{\mathbb{R}_{\ge 0}}

\newcommand{\GG}{{\mathcal{G}}}

\newcommand{\NN}{{\mathcal{N}}}
\newcommand{\LL}{{\mathcal{L}}}

\newcommand{\WW}{{\mathcal{W}}}

\newcommand{\QQ}{{\mathcal{Q}}}

\renewcommand{\SS}{{\mathcal{S}}}

\newcommand{\indicator}{\mathbf{1}}

\newcommand{\diag}[1]{\operatorname{diag}\left( #1\right)}

\newcommand{\Nin}{\NN^\text{in}}
\newcommand{\Nout}{\NN^\text{out}}

\newcommand{\timestep}{\Delta t}

\renewcommand{\epsilon}{\varepsilon}

\newcommand{\until}[1]{\{1,\dots, #1\}}

\newcommand{\setdef}[2]{\{#1 \, | \, #2\}}

\newcommand{\oprocendsymbol}{\hbox{$\bullet$}}
\newcommand{\oprocend}{\relax\ifmmode\else\unskip\hfill\fi\oprocendsymbol}

\newcommand{\longthmtitle}[1]{\mbox{}\textup{\textbf{(#1)}}}

\renewcommand{\theenumi}{(\roman{enumi})}

\begin{document}

\title{A general class of spreading processes with non-Markovian dynamics}
  
\author{Cameron Nowzari, Masaki Ogura, Victor M. Preciado, and George J. Pappas \thanks{The authors are
    with the Department of Electrical and Systems Engineering,
    University of Pennsylvania, Philadelphia, PA 19104, USA, {\tt\small
      \{cnowzari,ogura,preciado,pappasg\}@seas.upenn.edu}}}
\maketitle

\begin{abstract}
In this paper we propose a general class of models for spreading processes we call
the $SI^*V^*$ model. Unlike many works that consider a fixed number of compartmental
states, we allow an arbitrary number of states on arbitrary graphs with heterogeneous
parameters for all nodes and edges. As a result, this generalizes an extremely large number of
models studied in the literature including the MSEIV, MSEIR, MSEIS, SEIV, SEIR, SEIS,
SIV, SIRS, SIR, and SIS models. Furthermore, we show how the $SI^*V^*$ model
allows us to model non-Poisson spreading processes letting us capture much more complicated
dynamics than existing works such as information spreading through social networks or
the delayed incubation period of a disease like Ebola. This is in contrast to the overwhelming
majority of works in the literature that only consider spreading processes that can be
captured by a Markov process. After developing the stochastic model, we analyze its
deterministic mean-field approximation and provide conditions for when the disease-free
equilibrium is stable. Simulations illustrate our results. 
\end{abstract}

\section{Introduction}

The study of spreading processes on complex networks has recently gained
a massive surge of interest. With the wide range of applications including
the spreading of a computer virus, how a product is adopted by
a marketplace, or how an idea or belief is propagated through a social
network, it is no surprise that a plethora of different models and studies
have been devoted to this. However, an overwhelming majority of the 
stochastic models proposed and studied assume that transitions from one
state to another (e.g., a healthy individual recovering from a disease) is a Poisson
process that follows an exponential distribution. Unfortunately, this
simplifying assumption restricts the applicability of such models to
exclude a number of specific processes like how information is disseminated through Twitter
or how the Ebola virus is spreading in West Africa. 

In this paper we propose a very general class of epidemic models and show how
it can be used to study spreading processes that don't necessarily evolve according
to an exponential distribution. In addition to being able to account for non-Poisson
spreading processes, our model generalizes almost every model studied in the literature
including the MSEIV, MSEIR, MSEIS, SEIV, SEIR, SEIS, SIV, SIRS, SIR, and SIS models. 
The development and analysis of such a general model also allows rapid prototyping of 
future spreading processes that might not even exist today.

\subsection*{Literature review}

One of the oldest and most commonly studied spreading models is the 
Susceptible-Infected-Susceptible (SIS) model~\cite{WOK-AGM:27}. Early
works such as the one above often consider simplistic assumptions
such as all individuals in a population being equally likely to interact
with everyone else in the population~\cite{NTB:75}.
One of the first works to consider a continuous-time SIS model over arbitrary contact graphs using
mean field theory is~\cite{AL-JAY:76}, which provides conditions on when the disease-free
state of the system is globally asymptotically stable.

In addition to the simple SIS model, a myriad of different models have also been
proposed and studied in the literature. In~\cite{SF-EG-CW-VAAJ:09,NF:07}, the
authors add various states to model how humans might adapt their behavior
when given knowledge about the possibility of an emerging epidemic.
The work~\cite{FDS-FNC-CMS:12} considers the possible
effect of human behavior changes for the three state Susceptible-Alert-Infected-Susceptible
(SAIS) model. In~\cite{CN-VMP-GJP:15-TCNS}, a four-state generalized Susceptible-Exposed-Infected-Vigilant
(G-SEIV) model is proposed and studied. This model is appealing because it was shown to generalize a large
number of other models already studied in the literature~\cite{BAP-DC-MF-NV-CD:10,HWH:00}.
These models have been used to study the propagation of computer 
viruses~\cite{MMW-JL:03,MG-WG-DT:03} or products and information~\cite{DE-JK:10},
and of course the spreading of diseases~\cite{MEJN:02}. However, a large drawback
is that \emph{all} of the works above consider an underlying Markov process that
drives the system. 

While this may be well suited for a number of spreading processes,
they have also been applied in areas for which this is not a very good approximation. 
A notable example is the spreading of the Ebola virus. 
The work~\cite{GC-NWH-CCC-PWF-JMH:04} looks at the spreading of Ebola
in Congo and Uganda in 2004 and estimates the spreading properties of the virus
fitted to a four-state SEIR model. Similarly, the work~\cite{AK-MN-MD-MI:15} looks
at the more recent outbreak of Ebola in West Africa and again estimates the parameters
of the virus fitted to a six-state model. The larger number of states allows the model
to better capture things like human behavioral changes and also the incubation period of
the Ebola virus. However, just like all the works mentioned above, all transitions are assumed to
evolve according to an exponential distribution. More specifically, once an individual
is exposed to the virus at some time $t_0$, the probability that the individual has
started showing symptoms by time $t$ is given by $P(t) = 1 - e^{-\epsilon (t-t_0)}$ for
some $\epsilon > 0$. However, this is far from a good approximation when looking at the
empirical data collected over the years. The work~\cite{ME-SFD-NF:11} studies a certain
strain (Zaire) of the Ebola virus and concludes that the incubation period of the disease
is much better modelled as a log-normal distribution with a mean of 12.7 days and standard
deviation of 4.31 days, which cannot be well captured by the exponential distribution
above. Another prominent example today is the spreading of information
through social networks on the internet, such as Twitter or Digg. It has been observed multiple
times that the spreading of information in these networks is again better modelled 
as a log-normal distribution rather than an exponential one~\cite{KL-RG-TS:10,PVM-NB-CD:11,CD-NB-PVM:13}.

We are only aware of very few works that have considered spreading processes without
exponential distributions. The work~\cite{PVM-RVDB:13} studies the drastic effects
that non-exponential distributions can have both on the speed of the spreading and
on the threshold conditions for when the disease will die out or persist. A
simple SI model is studied in~\cite{HJ-JIP-KK-JK:14} without the complexity of an
arbitrary graph structure. The SIS model with general infection and recovery
times is considered in~\cite{EC-RVDB-PVM:13}. In this work we generalize this
idea to a much wider class of epidemic models.

\subsection*{Statement of contributions}

The contributions of this paper are threefold.
First, we propose the $SI^*V^*$ model that generalizes a very large
number of models studied in the literature today. Our model allows an arbitrary number of
`infected' and `vigilant' states unlike many models that have a fixed number of states.
Multiple infected states allows us to capture various stages of a disease or spreading process
which may have very different properties in terms of contagiousness, chance of recovery, etc. 
Multiple vigilant states allows us to capture different reasons that an individual might not be
susceptible to a disease including behavioral changes, vaccinations, or even death. Second,
we develop and analyze the deterministic mean-field approximation for the model and 
provide conditions for which the disease-free states of our model 
are globally asymptotically stable. Finally, we show how allowing our Markov model to have
an arbitrary number of states can be used to approximate non-Poisson spreading processes
with arbitrary accuracy. This allows us to much more accurately describe real-life phenomena,
such as the propagation of information through social networks or spreading of the Ebola virus,
which have recently been shown to evolve according to log-normal distributions rather than
exponential ones.

\subsection*{Organization}

We begin in Section~\ref{se:prelim} by reviewing some preliminary concepts that will be
useful in the remainder of the paper. We develop our proposed $SI^*V^*$ model
in Section~\ref{se:model} and analyze its stability properties in Section~\ref{se:analysis}.
In Section~\ref{se:phase} we show how our general class of models can be used to model
many stochastic spreading processes with state transition times that do not obey an
exponential distribution. We demonstrate the efficacy of our model and validate our stability
results through simulations of the spreading of the Ebola virus in Section~\ref{se:simulations}.
Finally, we gather our concluding remarks and ideas for future work in Section~\ref{se:conclusions}.

\section{Preliminaries}\label{se:prelim}

We denote by $\real$ and $\realnonnegative$ the sets of real and nonnegative real numbers, 
respectively. We define the indicator
function $\indicator_Z$ to be $1$ if $Z$ is true, and $0$ otherwise. 

\paragraph*{Graph theory}
Given a directed graph~$\GG$ with $N$ nodes, we denote by $A \in \real^{N \times N}$ 
the associated adjacency matrix. 
The components of $A$ are given by $a_{ji} = 1$ if and only if there exists
a directed edge from $i$ to $j$ on the graph~$\GG$. We denote the in-neighbors
and out-neighbors
of node $i$ as $\Nin_i = \setdef{j \in \until{N}}{a_{ij} = 1}$
and $\Nout_i = \setdef{j \in \until{N}}{a_{ji} = 1}$, respectively.
%
%Given square matrices $Q_1, \dots, Q_N$ where
%$Q_i \in \real^{n_i \times n_i}$, we let $\diag{Q_1, \dots, Q_N}$ denote the 
%$n \times n$ block diagonal matrix with $Q_1, \dots, Q_N$ on the diagonal where
%$n = \sum_{i=1}^N n_i$. 

Given a vector $q \in \real^n$, we let $\diag{q_1, \dots, q_n}$ denote
the $n \times n$ diagonal matrix with $q_1, \dots, q_n$ on the diagonal.
Given an arbitrary matrix $Q \in \real^{m \times n}$,
we define $\deg(Q) = \diag{\sum_{j=1}^n q_{1j}, \dots, \sum_{j=1}^n q_{mj} }$
the diagonal $m \times m$ matrix with row sums of $Q$ on the diagonal.
For a square matrix~$Q$, we define the Laplacian $\LL(Q) = \deg(Q) - Q$. 
A square matrix~$Q$ is \emph{Metzler} if its components $q_{ij} \geq 0$ for all
$i \neq j$. The following result will be useful in our analysis later. 

\begin{lemma}[Properties of Metzler matrices~\cite{AR:11,LF-SR:00}]\label{le:metzler}
Given a Metzler matrix $Q$, the following statements are equivalent:
\begin{enumerate}
\item $Q$ is Hurwitz 
\item There exists a positive vector $v$ such that $Qv < 0$
\item There exists a positive vector $w$ such that $w^T Q < 0$
\item $Q$ is nonsingular and $Q^{-1} \leq 0$
\end{enumerate}
\end{lemma}

%
%\begin{align*}
%f(t) = \frac{1}{\sqrt{2\pi} \sigma t} e^{ - \frac{ (\ln t - \mu)^2 }{2 \sigma^2} }
%\end{align*}
%$\mu$ mean 12.7 days, $\sigma$ SD 4.31 days

\section{Model description}\label{se:model}

We begin by formulating the $SI^*V^*$ that we study in the remainder of the paper. 
We follow the idea of the N-intertwined SIS model developed in~\cite{PVM-JO-RK:09} and
its extensions to the SAIS and generalized SEIV models developed in~\cite{FDS-FNC-CMS:12,CN-VMP-GJP:15-TCNS}. Although the latter models have been shown to generalize many different models
studied in the literature, they assume a fixed number of compartments for a given disease.
Instead, we build on a compartmental model studied in~\cite{MJK-PR:07} in which the
number of states relating to the disease are arbitrary. Unlike the population model (i.e., no
graph structure) studied
in~\cite{MJK-PR:07}, we are interested in proposing and analyzing this model
applied to arbitrary networks. In Section~\ref{se:phase} we show how this model can
be used to approximate a wide class of spreading processes on networks for \emph{any}
type of state transitions with arbitrary accuracy, rather than only Poisson processes (exponential
distributions).

We consider a virus spreading model with three classes of states called the $SI^*V^*$ model. 
The first class has only one state which is the susceptible state $S$. The susceptible state $S$ corresponds to a healthy individual who is capable of being exposed to the disease. The second class
has $m > 0$ states known as infectious states $I$. In the infectious class, an individual
can be in any of the $m$ states given by $I^k$ for $k \in \until{m}$. This allows the possibility
to model a number of variations to the infectious state including human behavior, severity of the disease, etc. The last class has $n > 0$ states known as vigilant states $V$. In
the vigilant class, an individual can be in any of the $n$ states given by $V^\ell$ for
$\ell \in \until{n}$. The vigilant class captures individuals who are not infected, but
also not immediately susceptible to contract the disease. The various states in the class
can be used to model different reasons that the individual is not susceptible such as
being vaccinated, having just recovered, or even deceased.

Consider a network with with $N$ nodes. For each node $i \in \until{N}$ we define the random variable $X_i(t) \in \{S, I^1, \dots, I^m, V^1, \dots, V^n\}$ as the state of node $i$ at a given time $t$. We consider a general directed contact graph~$\GG$ over which the disease can spread. A susceptible node
is only able to become exposed if it has at least one neighbor that is in any of the infectious states. 
A node $i$ can only be infected by nodes in $\Nin_i$ and can only infect nodes in $\Nout_i$. 

\renewcommand{\arraystretch}{1.3}
\begin{table}
\begin{center}
\begin{tabular}{|cl|}
\hline
$\delta^{k \ell}_i$ & Recovery rate from $I^k$ to $V^\ell$ \\
$\epsilon^{k k'}_i$ & Infection internal transition rate from $I^k$ to $I^{k'}$ \\
$\mu_i^{\ell \ell'}$ & Vigilant internal transition rate from $V^\ell$ to $V^{\ell'}$ \\
$\gamma^\ell_i$ & Rate of becoming susceptible from $V^\ell$ to $S$ \\
$\theta^\ell_i$ & Rate of becoming vigilant from $S$ to $V^\ell$ \\
$\beta_{ij}^k$ & Infection rate due to infected $(I^k)$ neighbor $j$ from $S$ to $I^1$ \\
\hline
\end{tabular}
\end{center}
\caption{Parameter definitions}\label{ta:definition}
\end{table}

The compartmental Markov process is defined by the following parameters. 
Let $\delta^{k \ell}_i$ be the recovery rate of node $i$ going from infectious
state $I^k$ to vigilant state $V^\ell$. This allows the possibility to model
different recovery rates depending on which state of an infection the individual
is in and which vigilant state the individual will end up in. We let
$D_i = [\delta^{k \ell}_i]_{k \ell} \in \real^{m \times n}$ be the matrix that describes these transitions. 
Let $\epsilon_i^{k k'}$ be the rate at which an individual in infectious state $I^k$
moves to infectious state $I^{k'}$. This
can model the various different stages or severity of a disease and how individuals
move from one stage to another. We let $E_i = [\epsilon_i^{k k'}]_{k k'} \in \real^{m \times m}$ be the matrix
that describes these transitions. We denote by $\mu_i^{\ell \ell'}$ the internal
transition rate from vigilant state $V^\ell$ to $V^{\ell'}$. This can model different
levels and types of vigilance in individuals, such as behavioral changes, changing
medications/vaccines, etc. We let $M_i = [\mu^{\ell \ell'}_i]_{\ell \ell'} \in \real^{n \times n}$ be the matrix
that describes these transitions.
We denote by $\gamma_i^\ell$ and $\theta_i^\ell$ the rates of moving from $V^\ell$
to $S$ and $S$ to $V^\ell$, respectively. Finally, let $\beta_{ij}^k$ be the
effect that a neighbor $j \in \Nin_i$ of node $i$ in state $I^k$ has on $i$.
The rate that an individual~$i$ moves from $S$ to $I^1$ is then given by
\begin{align*}
\sum_{k'=1}^m \sum_{j \in \Nin_i} \beta_{ij}^{k'} Y_j^{k'}, 
\end{align*}
where
\begin{align}\label{eq:indicator}
Y_j^k(t) &= \indicator_{X_j(t) = I^k}.
\end{align}
Note that when a susceptible individual~$i$ becomes infected, it always moves into the first
infectious state~$I^1$. It is then free to move to the other infectious stages
according to~$E_i$. 
All the disease parameters are nonnegative. Table~\ref{ta:definition} presents
the definitions of these parameters for convenience. 

The dynamics of the
epidemic spread is then modeled using the definition of the infinitesimal generator from~\cite{PVM:09}. For brevity, we only write out a small subset of the possible transitions,
\begin{align*}
P(X_i(t') \hspace*{-.4mm} = I^1 | X_i(t) \hspace*{-.4mm} = S, X(t)) \hspace*{-.4mm} &= \sum_{k'=1}^m \sum_{j \in \Nin_i} \beta_{ij}^{k'} Y_j^{k'} + o , \\
P(X_i(t') \hspace*{-.4mm} = I^2 | X_i(t) \hspace*{-.4mm} = I^1, X(t)) \hspace*{-.4mm} &= \epsilon_i^{12} + o , \\
P(X_i(t') \hspace*{-.4mm} = I^1 | X_i(t) \hspace*{-.4mm} = I^2, X(t)) \hspace*{-.4mm} &= \epsilon_i^{21} + o, \\
P(X_i(t') \hspace*{-.4mm} = V^\ell | X_i(t) \hspace*{-.4mm} = I^1, X(t)) \hspace*{-.4mm} &= \delta^{1\ell}_i + o , 
\end{align*}
where $t' = t + \timestep$ and $o = o(\timestep)$. 

Figure~\ref{fig:siv} shows the $(1+m+n)$-state $SI^*V^*$ compartmental Markov model 
for a single node~$i$. Note that the only graph-based transition is from
the susceptible state~$S$ to the first infected state~$I^1$. 
The state of the entire network $X(t)$ then lives in a $(1+m+n)^N$ dimensional space
making it very hard to analyze directly. Instead, we utilize a mean-field
approximation to reduce the complexity of the entire system. We do this by
replacing $Y_j^k$ in~\eqref{eq:indicator} by its expected value $E[Y_j]^k$.

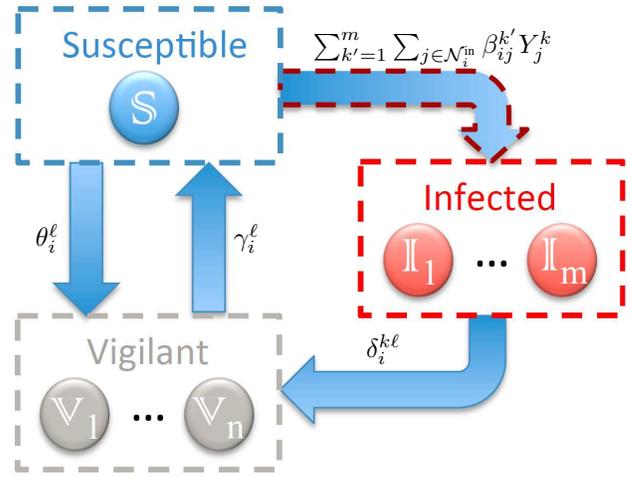
\begin{figure}[tb]
\centering
\includegraphics[width=.95\linewidth]{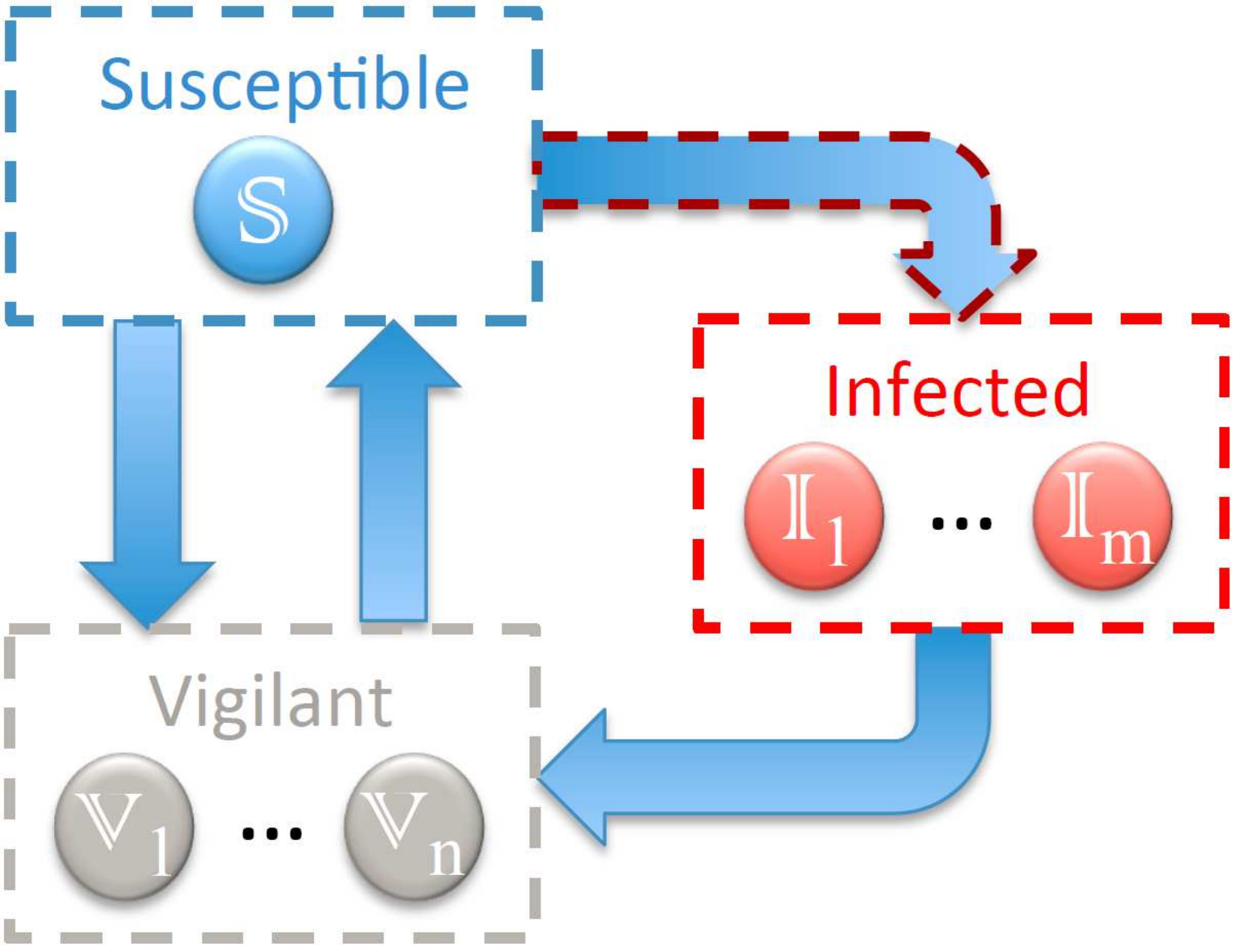}
\put(-120,160){$\sum_{k'=1}^m \sum_{j \in \Nin_i} \beta_{ij}^{k'} Y_j^k$}
\put(-100,45){$\delta_i^{k\ell}$} \put(-150,87){$\gamma_i^\ell$}
\put(-225,87){$\theta_i^\ell$}
\caption{$(1+m+n)$-state $SI^*V^*$ compartmental model for node~$i$.
Internal transition rates among the infected and
vigilant classes are not shown.}\label{fig:siv}
\end{figure}

We denote by $\left[ S_i(t), I_i^1(t), \dots, I_i^m(t), V_i^1(t), \dots, V_i^n(t) \right]^T$
%\in \real^{1+m+n}$ 
the probability vector associated with
node $i$ being in each of these states, i.e.,
\begin{align}\label{eq:stateconstraint}
S_i(t) + \sum_{k'=1}^m I_i^{k'}(t) + \sum_{\ell' = 1}^n V_i^{\ell'}(t) = 1, \\
S_i(t), I_i^k(t), V_i^\ell(t) \geq 0, \notag
\end{align}
for all $i \in \until{N}$, $k \in \until{m}$, and $\ell \in \until{n}$.

The $S I^* V^*$ model we consider in this paper is then given by
\begin{align}\label{eq:continuousdynamics}
\dot{S}_i(t) &= \sum_{\ell'=1}^n \gamma_i^{\ell'} V_i^{\ell'} - \theta_i^{\ell'} S_i - S_i \sum_{k'=1}^m \sum_{j \in \Nin_i} \beta_{ij}^{k'} I_j^{k'} , \notag \\
\dot{I}_i^1(t) &= S_i \sum_{k'=1}^m \sum_{j \in \Nin_i} \beta_{ij}^{k'} I_j^{k'} - I_i^1 \sum_{\ell'=1}^n \delta^{1\ell'}_i \notag \\
& \hspace*{14ex} + \sum_{k'=1}^m I_i^{k'} \epsilon^{k'1}_i - I_i^1 \epsilon_i^{1k'}  , \\ 
\dot{I}_i^k(t) &=  - I_i^k \sum_{\ell'=1}^n \delta^{k\ell'}_i + \sum_{k'=1}^m I_i^{k'} \epsilon^{k'k}_i - I_i^{k} \epsilon_i^{kk'} , \notag \\
\dot{V}_i^\ell(t) &= \sum_{k'=1}^m \delta^{k'\ell} I_i^{k'} + \theta^\ell_i S_i - \gamma^\ell_i V_i^\ell + \sum_{\ell' = 1}^n V_i^{\ell'} \mu_i^{\ell' \ell} - V_i^\ell \mu_i^{\ell \ell'} , \notag
\end{align}
for $k \in \{ 2, \dots, m \}$ and $\ell \in \until{n}$. 

Due to constraints~\eqref{eq:stateconstraint}, one of the equations~\eqref{eq:continuousdynamics} is
redundant. By setting $S_i(t) = 1 - \sum_{k'=1}^m I_i^{k'}(t) - \sum_{\ell'=1}^n V_i^{\ell'}(t)$, we can describe the system by
\begin{align}\label{eq:removed}
\dot{I}_i^1(t) &= (1 - \sum_{k'=1}^m I_i^{k'} - \sum_{\ell'=1}^n V_i^{\ell'}) \sum_{k'=1}^m \sum_{j \in \Nin_i} \beta_{ij}^{k'} I_j^{k'} \notag \\
& \hspace*{3ex}  - I_i^1 \sum_{\ell'=1}^n \delta^{1\ell'}_i + \sum_{k'=1}^m I_i^{k'} \epsilon^{k'1}_i - I_i^1 \epsilon_i^{1k'} , \notag \\ 
\dot{I}_i^k(t) &= - I_i^k \sum_{\ell=1}^n \delta^{k\ell}_i + \sum_{k'=1}^m I_i^{k'} \epsilon^{k'k}_i - I_i^{k} \epsilon_i^{kk'} ,\\
\dot{V}_i^\ell(t) &= \sum_{k'=1}^m \delta^{k'\ell}_i I_i^{k'} + \theta^\ell_i (1 - \sum_{k'=1}^m I_i^{k'} - \sum_{\ell'=1}^n V_i^{\ell'}) \notag \\
& \hspace*{3ex} - \gamma^\ell_i V_i^\ell + \sum_{\ell' = 1}^n V_i^{\ell'} \mu_i^{\ell' \ell} - V_i^\ell \mu_i^{\ell \ell'} . \notag
\end{align}
Next, we are interested in studying the stability properties for this set of $N(m+n)$ ODEs. 

\section{Stability analysis of $SI^*V^*$}\label{se:analysis}

Let $\mathbf{x}_i = \left[ I_i^1, \dots, I_i^m \right]^T$, $\mathbf{y}_i = \left[ V_i^1, \dots, V_i^n \right]^T$, $\mathbf{x} = \left[ \mathbf{x}_1^T, \dots, \mathbf{x}_m^T \right]^T$,
and $\mathbf{y} = \left[ \mathbf{y}_1^T, \dots, \mathbf{y}_n^T \right]^T$. Naturally,
we are interested in conditions that will drive the system to a disease-free state,
i.e., $\mathbf{x} \rightarrow 0$. We begin by writing the dynamics of the system
as
\begin{align}\label{eq:matrixeqn}
\left[ \begin{array}{c} \mathbf{\dot{x}} \\ \mathbf{\dot{y}} \end{array} \right] = \WW \left[ \begin{array}{c} \mathbf{x} \\ \mathbf{y} \end{array} \right] + H,
\end{align}
where $\WW \in \real^{N(n+m) \times N(n+m)}$ captures the linear part of the dynamics
and $H = [H_x^T, H_y^T]^T$, with $H_x \in \real^{Nm}$ and $H_y \in \real^{Nn}$, captures 
the nonlinear part.
For convenience, we split $\WW$ into smaller matrices such that
\begin{align*}
\WW = \left[ \begin{array}{cc} \WW_{xx} & \WW_{xy} \\ \WW_{yx} & \WW_{yy} \end{array} \right] .
\end{align*}
First, we define the matrix $\WW_{xx} \in \real^{Nm \times Nm}$ that describes how infected states affect other infected states.
Let $\WW_{xx} = [\WW_{xx}^{ij}]$ be the block matrix where

\begin{align*}
\WW_{xx}^{ii} = -\LL(E_i) - \operatorname{deg}(D_i)
\end{align*} 
%with
%\begin{align*}
%\operatorname{deg}(D_i) = \diag{ \sum_{\ell'=1}^n \delta_i^{1 \ell'} , \dots, \sum_{\ell'=1}^n \delta_i^{m \ell'}} ,
%\end{align*}
describes the internal transitions between the infected states
and the transitions from all infected states to all vigilant states, and
\begin{align*}
\WW^{ij}_{xx} = 
\left[ \begin{array}{c} 
\beta_{ij}^1, \dots, \beta_{ij}^m \\
\mathbf{0}_{(m-1) \times m}
\end{array} \right]
\end{align*}
describes the (linear) transitions of node~$i$ from the susceptible state to 
the first infected state $I^1$ due to other nodes. Second, it is easy to see from~\eqref{eq:removed}
that $\WW_{xy} = \mathbf{0}_{Nm \times Nn}$ because the infected states
are not (linearly) affected by the vigilant states. 

Third, we define the matrix $\WW_{yx}$ that describes how infected
states affect the vigilant states. Since all these transitions happen
internally (i.e., do no depend on the network structure), $\WW_{yx} = [\WW_{yx}^{ii}] \in \real^{Nn \times Nm}$
is a block diagonal matrix where %$\WW_{yx}^{ii} \in \real^{n \times m}$ is given by
\begin{align*}
\WW_{yx}^{ii} = D_i^T - \left[ \begin{array}{ccc} 
\theta_i^1 & \dots & \theta_i^1 \\
\theta_i^2 & \dots & \theta_i^2 \\
\vdots & & \vdots \\
\theta_i^n & \dots & \theta_i^n \end{array} \right]
\end{align*}
describes the transitions from all infected states to all vigilant states
and the transitions from the susceptible state to all vigilant states.

Finally, we define the matrix $\WW_{yy} \in \real^{Nn \times Nn}$ that describes how vigilant
states affect other vigilant states. As before, since all these transitions
happen internally, $\WW_{yy} = [\WW_{yy}^{ii}]$
is a block diagonal matrix where
\begin{align*}
\WW_{yy}^{ii} = -\LL(M_i) - \left[ \begin{array}{ccc} 
\theta_i^1 & \dots & \theta_i^1 \\
\theta_i^2 & \dots & \theta_i^2 \\
\vdots & & \vdots \\
\theta_i^n & \dots & \theta_i^n \end{array} \right] - \diag{\gamma_i^1, \dots, \gamma_i^m}
\end{align*}
describes the internal transitions between the vigilant states, transitions from the susceptible
state to all vigilant states, and transitions from
all vigilant states to the susceptible state.
 
We now define the column vector~$H$. 
As can be seen in equation~\eqref{eq:removed}, the nonlinearities only enter into the dynamics
of the first infectious state $I^1$. Thus, we can describe
$H_x^i \in \real^{m}$ as
\begin{align*}
 \left[ \begin{array}{c} 
\left( - \sum_{k'=1}^m I_i^{k'} - \sum_{\ell'=1}^n V_i^{\ell'} \right) \sum_{k'=1}^m \sum_{j \in \Nin_i} \beta_{ij}^{k'} I_j^{k'} \\
\mathbf{0}_{(m-1) \times 1} \end{array} \right] .
\end{align*} 
Due to removing the susceptible state, we also have a constant forcing given by
$H_y^i = \left[ \theta^1_i, \dots, \theta^n_i \right]^T \in \real^n$. 

\begin{theorem}\longthmtitle{Sufficient condition for global stability of disease-free equilibrium}\label{th:global-stability}
The disease-free states of~\eqref{eq:continuousdynamics} are globally
asymptotically stable if~$\WW_{xx}$ is Hurwitz.
\end{theorem}
\begin{proof}
We begin by noticing that $\WW_{xx}$ is a Metzler matrix. Using Lemma~\ref{le:metzler}
and the condition that $\WW_{xx}$ is Hurwitz,
we know there exists a positive vector $v \in \realpositive^{Nm}$ such that $v^T \WW_{xx} < 0$.
Consider the Lyapunov function
\begin{align*}
J = v^T \mathbf{x},
\end{align*}
then
\begin{align*}
\dot{J} &= v^T \mathbf{\dot{x}} \\
&= v^T \WW_{xx} \mathbf{x} + v^T H_x \\
& \leq v^T \WW_{xx} \mathbf{x}
\end{align*}
is strictly negative for all $\mathbf{x} \neq 0$. It can then easily be
shown using LaSalle's Invariance Principle that $\mathbf{x} \rightarrow 0$~\cite{HKK:02}.
\end{proof}

Note that this conservative result essentially ignores the vigilant class 
and how it can help a disease die out. In other words, this result states
that if the vigilant class is removed and the disease-free state is still
globally asymptotically stable, adding the vigilant class cannot hurt. For this
reason, we do not see any forms of the parameters $\mu_i^{\ell \ell'}$,
$\gamma_i^\ell$, or $\theta_i^\ell$ appear in the condition.

In order to find the necessary and sufficient condition, we linearize
the entire system around the disease-free equilibrium. For simplicity,
we assume here that there exists no absorbing state inside the vigilant
class. Note that if there exist any absorbing states in the vigilant class,
it does not make sense to discuss asymptotic properties anyway, as all individuals
will eventually end up there, resulting in a disease-free equilibrium. 

We begin by computing the unique equilibrium for the vigilance class states. 
Letting $\mathbf{\dot{y}} = 0$ and $\mathbf{x} = 0$, we get
\begin{align*}
\mathbf{y}^* = -\WW_{yy}^{-1} H_y . 
\end{align*}
The inverse of~$\WW_{yy}$ is guaranteed to exist because it is a block diagonal
matrix made up of negative definite matrices. 
Let $\left[\bar{V}_i^1, \dots, \bar{V}_i^n \right]^T = \mathbf{y}_i^*$ for all $i \in \until{N}$.
The linearization of~\eqref{eq:matrixeqn} around the point $\mathbf{x} = 0, \mathbf{y} = \mathbf{y}^*$
is then given by
\begin{align*}
\left[ \begin{array}{c} \mathbf{\dot{x}} \\ \mathbf{\dot{y}} \end{array} \right] = \QQ \left[ \begin{array}{c} \mathbf{x} \\ \mathbf{y} \end{array} \right] ,
\end{align*}
where
\begin{align*}
\QQ_{xx}^{ii} = - \LL(E_i) - \deg(D_i),
\end{align*}
\begin{align*}
\QQ^{ij}_{xx} = 
\left( 1 - \sum_{\ell'=1}^n \bar{V}_i^{\ell'} \right) \left[ \begin{array}{c} 
 \beta_{ij}^1, \dots, \beta_{ij}^m \\
\mathbf{0}_{(m-1) \times m}
\end{array} \right] ,
\end{align*}
$\QQ_{xy} = \WW_{xy} = \mathbf{0}_{Nm \times Nn}$, $\QQ_{yx} = \WW_{yx},$
and $\QQ_{yy} = \WW_{yy}$. 

\begin{theorem}\longthmtitle{Local stability of disease-free equilibrium}\label{th:local-stability}
The disease-free states of~\eqref{eq:continuousdynamics} are locally asymptotically
stable if and only if $\QQ_{xx}$ is Hurwitz. 
\end{theorem}
\begin{proof}
Since the linear dynamics of the infectious states~$\mathbf{x}$ do not depend on the
vigilant states~$\mathbf{y}$, we can write
\begin{align*}
\mathbf{\dot{x}} = \QQ_{xx} \mathbf{x} .
\end{align*}
It is then clear that $\mathbf{x} \rightarrow 0$ in this linearization of the original system 
if and only if~$\QQ_{xx}$ is Hurwitz.
\end{proof}
%\margin{need to show this inverse always exists and solutions are all between 0 and 1...}

\begin{remark}\longthmtitle{Global stability of disease-free equilibrium}
{\rm
We conjecture, and verify through simulation, that the necessary and sufficient
local stability result of Theorem~\ref{th:local-stability} is indeed a global
result but it has yet to be shown. We expect to have its proof completed in the
near future. \oprocend }
\end{remark}

We note here that based on these results, determining the stability properties of
the disease-free equilibrium amounts to checking if an $Nm \times Nm$ matrix is Hurwitz,
even though the original system is $N(m+n)$-dimensional. 

\section{Non-exponential distributions}\label{se:phase}

In this section we show how existing epidemic models can be studied
with non-Markovian state transitions by appropriately constructing an instance
of our $SI^*V^*$ model. To simplify the exposition, we demonstrate this idea for
a single example, the spreading of the Ebola virus, but note that the same idea 
can be used to expand a very large number of different models.

The underlying model we consider is the four-state G-SEIV model proposed in~\cite{CN-VMP-GJP:15-TCNS}
and shown in Figure~\ref{fig:seiv}. The `Susceptible' state $S$ captures individuals who
are able to be exposed to the disease, the `Exposed' state $E$ captures individuals who
have been exposed to the disease but have not yet developed symptoms, the `Infected' state $I$
captures individuals who are displaying symptoms and contagious, and the `Vigilant' state $V$
captures individuals who are not immediately susceptible to the disease (e.g., just recovered,
quarantining themselves, deceased).

%\begin{figure}
%  \centering
%  \scalebox{0.75}{\begin{tikzpicture}[shorten >= 1pt,node distance=2.25cm,auto]
%  \pgfsetlinewidth{2pt}
%    \node[state, fill=green, accepting] (S)              {$S$};
%    \node (G) [right=of S] {};
%    \node[state, fill=yellow, accepting] (E) [right=of G] {$E$};
%    \node[state, fill=blue!60!white, accepting] (V) [below=of G] {$V$};
%    \node[state, fill=red, accepting] (I) [right=of V] {$I$};
%    
%    \path[->] (S) edge [bend left] node {${\theta}_i$} (V)
%    (V) edge [bend left] node {$\gamma_i$} (S)
%    (I) edge node {${\delta}_i$} (V)
%    (E) edge [bend left] node {} (I);
%
%    \path[->,red,dashed] (S) edge node{$\sum_{j = 1}^N {\beta}^E_{ji} E_j + \beta^I_{ji} I_j$} (E);
%    \end{tikzpicture}}\caption{Four-state G-SEIV compartmental model}\label{fig:seiv}
%\end{figure}

\begin{figure}[tb]
\centering
\includegraphics[width=.95\linewidth]{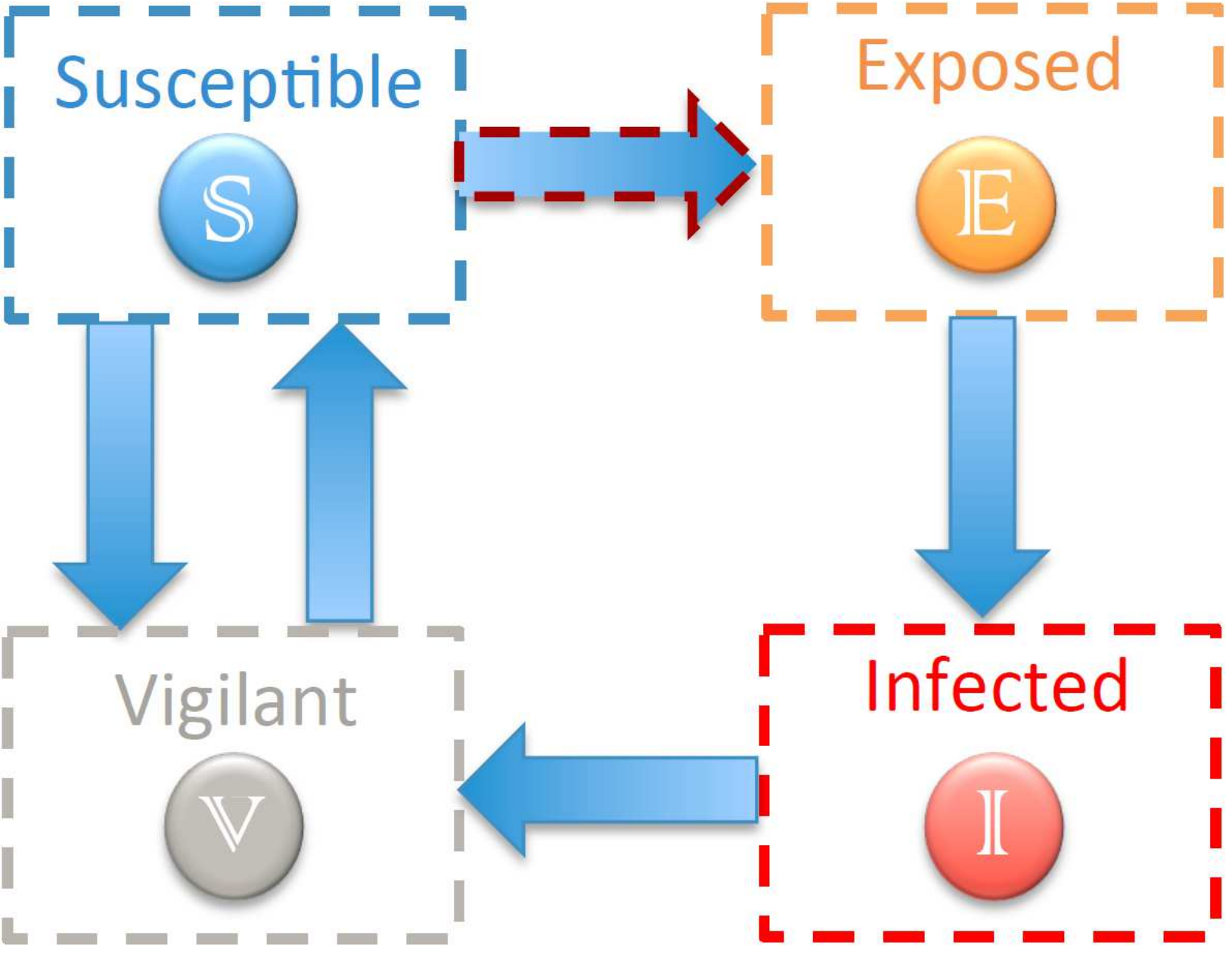}
\put(-143,170){\scalebox{.9}{$\sum_{j \in \Nin_i} \beta_{ij}^{I} Y_j^I$}}
\put(-110,45){$\delta_i$} \put(-160,87){$\gamma_i$}
\put(-230,87){$\theta_i$} 
\caption{Four-state G-SEIV compartmental model for node~$i$.}\label{fig:seiv}
\end{figure}
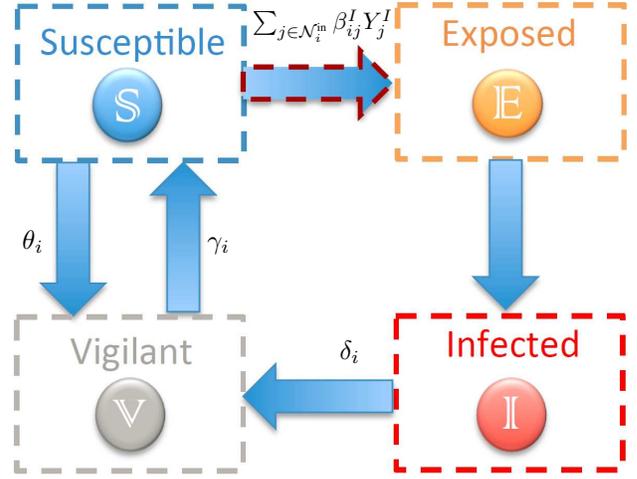

We assume that all transitions are Poisson processes \emph{except}
the transition from exposed to infected. For this transition we use a log-normal distribution with mean 12.7 days and standard deviation 4.31 days as proposed in~\cite{ME-SFD-NF:11} based on empirical data. 
In the construction of the approximation, an important role is played by the class of probability distributions called phase-type distributions~\cite{SA-ON-MO:96}.

Consider a time-homogeneous Markov process in continuous-time with $p+1$ ($p\geq 1$) states such that the states~$1$, $\dotsc$, $p$ are transient and the state $p+1$ is absorbing. The infinitesimal generator of the process is then necessarily of the form
\begin{equation}\label{eq:infgen}
\begin{bmatrix}
\SS & -\SS \mathbf{1}_{p \times 1}\\
\mathbf{0}_{1 \times p} & 0
\end{bmatrix},
\end{equation}
where $\SS \in \mathbb{R}^{p\times p}$ is an invertible Metzler matrix with non-positive row-sums. Also let
\begin{equation}\label{eq:initial}
\begin{bmatrix}
\phi \\ 0
\end{bmatrix}  \in \mathbb{R}^{p+1},\ \phi \in \realnonnegative^p
\end{equation}
be the initial distribution of the Markov process. Then, the time to absorption into the state~$p+1$ of the Markov process, denoted by $(\phi, \SS)$, is called a \emph{phase-type distribution}. It is known that the set of phase-type distributions is dense in the set of positive valued distributions~\cite{DRC:55}. Moreover, there are efficient fitting algorithms to approximate a given arbitrary distribution by a phase-type distribution~\cite{SA-ON-MO:96}.

We now show how this can be used to expand the G-SEIV model
to an instance of our $SI^*V^*$ model such that the time
it takes from to reach the infected state from the exposed state
follows a phase-type distribution.

\begin{proposition}\label{prop:ph}
Consider the $SI^*V^*$ model with $m = p+1$ infectious states, where $I^m$ corresponds
to the infected state and $I^k$ for $k \in \until{p}$ correspond to the exposed state.
Let $\SS_I \in \real^{m \times m}$ be given by
\begin{equation}
[\SS_I]_{kk'} = 
\begin{cases}
\epsilon^{kk'}&\text{$k\neq k'$} , 
\\
-\sum_{k'=1}^m \epsilon^{kk'} &\text{otherwise} .
\end{cases} 
\end{equation}
Then the length of time that it takes a node $i$ to go from state $I^1$ to $I^m$
follows the distribution $(e_1, S_I)$,
where $e_1 = \left[1, \mathbf{0}_{1 \times p} \right]^T$. 
%Also define $S_V \in \mathbb{R}^{n\times n}$ by
%\begin{equation}
%[S_V]_{\ell\ell'} = 
%\begin{cases}
%\mu^{\ell\ell'}&\text{$\ell\neq \ell'$}
%\\
%- \gamma^{\ell} -\sum_{\ell'=1}^m \mu^{\ell\ell'} &\text{otherwise}
%\end{cases}
%\end{equation}
%Then the length of time that a node $i$ stays vigilant follows the  distribution $(e_1, S_V)$.
\end{proposition}

\begin{proof}
Without loss of generality, assume that a node $i$ becomes exposed at time $t = 0$,
i.e., $X_i(0) = I^1$. Letting $T$ be the time that node $i$ first reaches the infected
state, i.e., $X_i(T) = I^{m}$, we need to show that $T$ follows $(e_1, \SS_I)$. 

We begin by noticing that the initial condition $X_i(0) = e_i$ corresponds to
node~$i$ having just entered the exposed state. This is because in our $SI^*V^*$ model,
a susceptible node can only enter the infected states through $I^1$. Also, from the definition of the model, it is clear that the Markov process $X_i$ has the infinitesimal generator~\eqref{eq:infgen} on the time interval $[0, T]$. Also $I^{m}$ is clearly the only absorbing state of the process (when confined on $[0, T]$). The above observation shows that $T$ follows the phase-type distribution $(e_1, \SS_I)$ because we have satisfied the
definitions of the infinitesimal generator~\eqref{eq:infgen} and initial distribution~\eqref{eq:initial}.
\end{proof}

Proposition~\ref{prop:ph} shows that it is possible to model the transition from the exposed state to
the infected state of the SEIV model as a phase-type distribution. This is done by essentially
expanding the exposed state in the original SEIV model from a single state to~$p$. 
As noted earlier, this is only one
specific example that can be extended to model many different state transitions as phase-type distributions
rather than exponential ones. Next, we show how an arbitrary distribution can be approximated as
a phase-type distribution and how to choose the parameters for our $SI^*V^*$ model to realize
the desired distribution.

%Since the parameters $\epsilon^{kk'}$, $\delta^{k1}$, $\mu^{\ell\ell'}$, and $\gamma^\ell$ can be arbitrary nonnegative numbers, Proposition~\ref{prop:ph} shows that the $SI^*V^*$ model can express any infectious process on networks with recovery or vigilance times having the distribution of the form $(e_1, S)$, where $S$ is any stable and Metzler matrix {\color{red}[Why ``any stable and Metzler matrix''? More explanation?]}. 

Continuing with our Ebola example, we show how phase-type distributions can well approximate the log-normal distribution of Ebola's incubation time with mean of $12.7$ days and a standard-deviation of $4.31$ days. 
To do this, we utilize the expectation-maximization algorithm proposed in~\cite{SA-ON-MO:96} 
%(and available at~\mbox{\small\url{http://home.math.au.dk/asmus/pspapers.html}}). 
Figure~\ref{fig:p246810} shows the results for different amounts of internal states~$p$. 
We can see here that the more internal states~$p$ we use, the closer our phase-type distribution
becomes to the desired log-normal distribution. An instance of the phase-type distribution
for~$p=10$ is shown in Figure~\ref{fig:phase}.
\begin{figure}[tb]
\centering
\includegraphics[width=.95\linewidth]{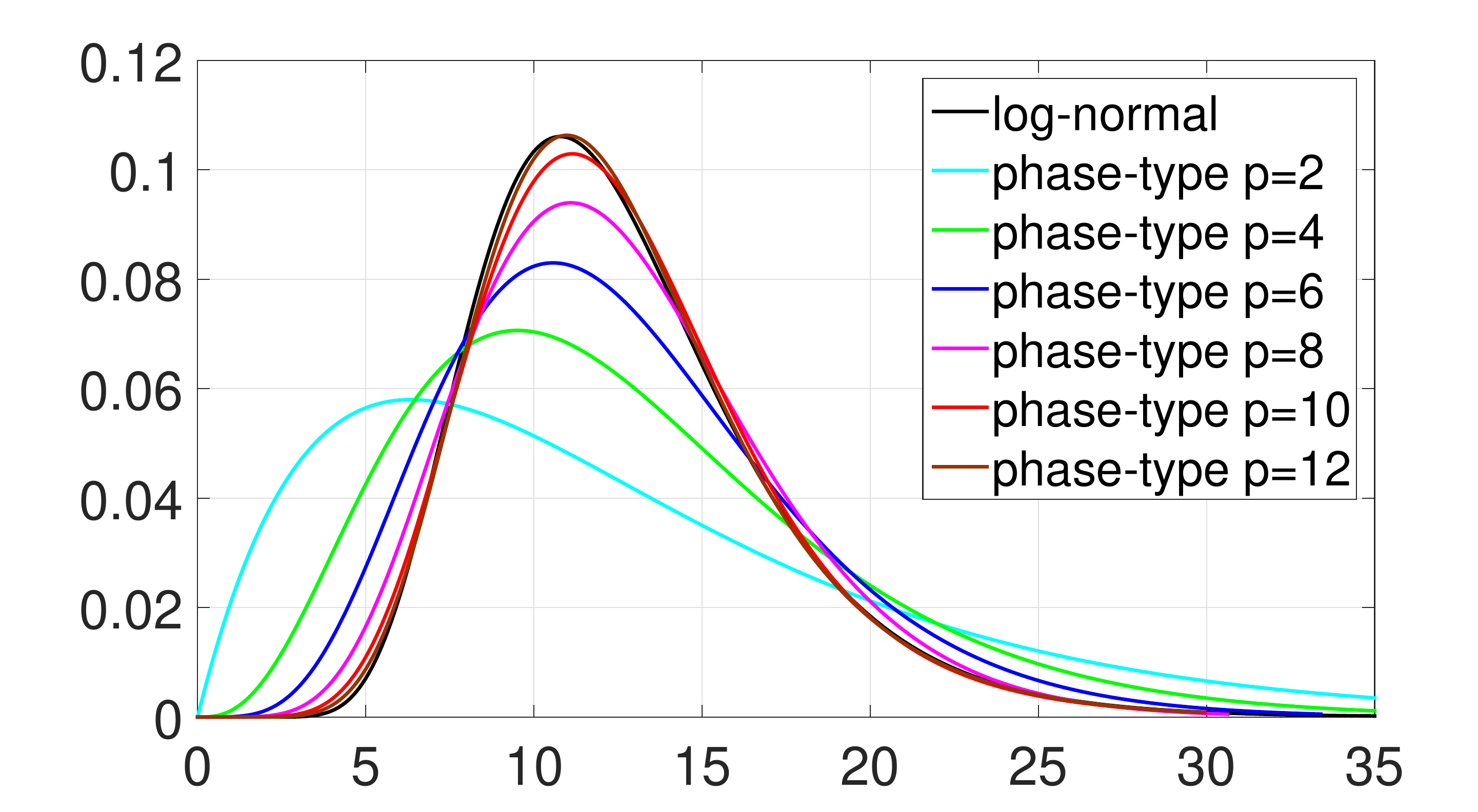}
\caption{Approximation of log-normal distribution (mean 12.7 days and standard deviation 4.31 days) with phase-type distributions for Ebola's incubation period~\cite{ME-SFD-NF:11}.}% $\mu = 2.4871$ and $\sigma = 0.3302$.}
\label{fig:p246810}
\end{figure}

\begin{remark}\label{re:dense}
{\rm
In fact one can show that the set of the phase-type distributions of the form $(e_1, \SS)$ is dense in the set of all the phase-type distributions as follows. Let $(\phi, \SS)$ be a given phase-type distribution. For a positive real number $r$ define the block-matrix
\begin{equation}\label{eq:T_n}
T_r = \begin{bmatrix}
-r & r\phi^\top
\\
0 & \SS
\end{bmatrix}. 
\end{equation}
Then we can prove that the sequence of the phase type distributions $\{(e_1, T_r)\}_{n=1}^\infty$ 
with $p + 2$ states
converges to the phase type distribution $(\phi, \SS)$ with $p+1$ states. Here we provide only an outline of the proof. Let $X$ be the time-homogeneous Markov process having the infinitesimal generator~\eqref{eq:T_n} and the initial distribution $e_1$. Define $t_1 = \sup\{\tau : X(\tau) = 1\}$, i.e.,
the time at which the Markov process leaves the first state, and let $t_2$ be the time the process $X$ is absorbed into the state $p+2$. By the definition of the first row of the generator $T_r$, we see that $X(t_1)$ follows the distribution $\left[ 0 , \phi^T, 0 \right]^T \in \realnonnegative^{p+2}$.

Therefore $t_2 -t_1$ follows $(\phi, \SS)$. Moreover, as $r$ increases, the probability density function of $t_1$ converges to the Dirac delta on the point $0$. Therefore the random variable $t_2 = t_1 + (t_2 -t_1)$, 
which follows $(e_1, T_r)$ by definition, converges to $(\phi, \SS)$. The details are omitted due to
space restrictions. \oprocend }
\end{remark}

The implications of Remark~\ref{re:dense} are that although our $SI^*V^*$ model only
allows a susceptible node to enter the infected class through $I^1$, we are still able to
model \emph{any} phase-type distribution $(\phi, \SS)$ rather than just $(e_1, \SS)$. 

\section{Simulations}\label{se:simulations}

Here we demonstrate how the results of Section~\ref{se:phase} can be used to model
a spreading process with a non-exponential state transition and validate the stability
results of Section~\ref{se:analysis} by simulating the spreading of Ebola.
The underlying model we use is a four-state G-SEIV model proposed in~\cite{CN-VMP-GJP:15-TCNS}
and shown in Figure~\ref{fig:seiv}. The `Susceptible' state $S$ captures individuals who
are able to be exposed to the disease, the `Exposed' state $E$ captures individuals who
have been exposed to the disease but have not yet developed symptoms, the `Infected' state $I$
captures individuals who are displaying symptoms and contagious, and the `Vigilant' state $V$
captures individuals who are not immediately susceptible to the disease (e.g., just recovered,
quarantining themselves).

We assume that all transitions are Poisson processes \emph{except}
the transition from $E$ to $I$. For this transition we use a log-normal distribution with mean 12.7 days and standard deviation 4.31 days as proposed in~\cite{ME-SFD-NF:11} based on empirical data. Using the EM algorithm
proposed in~\cite{SA-ON-MO:96} with $p=10$ phases, we expand the exposed state from a single state to
10 states. This means we can describe our four state non-Poisson SEIV model by a 13-state Poisson $SI^*V^*$
model with one susceptible state, one vigilant state, and $m = 11$ infectious states where $I^{11}$
is the only contagious state. The other infectious states $I^k$ for $k \in \until{10}$ correspond
to the exposed (but not symptomatic) state of the original SEIV model. 
The obtained internal compartmental model between the 
exposed state and infected state is shown in Figure~\ref{fig:phase}. 

{
\psfrag{E}[cc][cc]{$I^1$}
\psfrag{P2}[cc][cc]{$I^2$}
\psfrag{P3}[cc][cc]{$I^3$}
\psfrag{P4}[cc][cc]{$I^4$}
\psfrag{P5}[cc][cc]{$I^5$}
\psfrag{P6}[cc][cc]{$I^6$}
\psfrag{P7}[cc][cc]{$I^7$}
\psfrag{P8}[cc][cc]{$I^8$}
\psfrag{P9}[cc][cc]{$I^9$}
\psfrag{P10}[cc][cc]{$I^{10}$}
\psfrag{I}[cc][cc]{$I^{11}$}
\begin{figure*}[htb!]
\centering
\includegraphics[width=14cm]{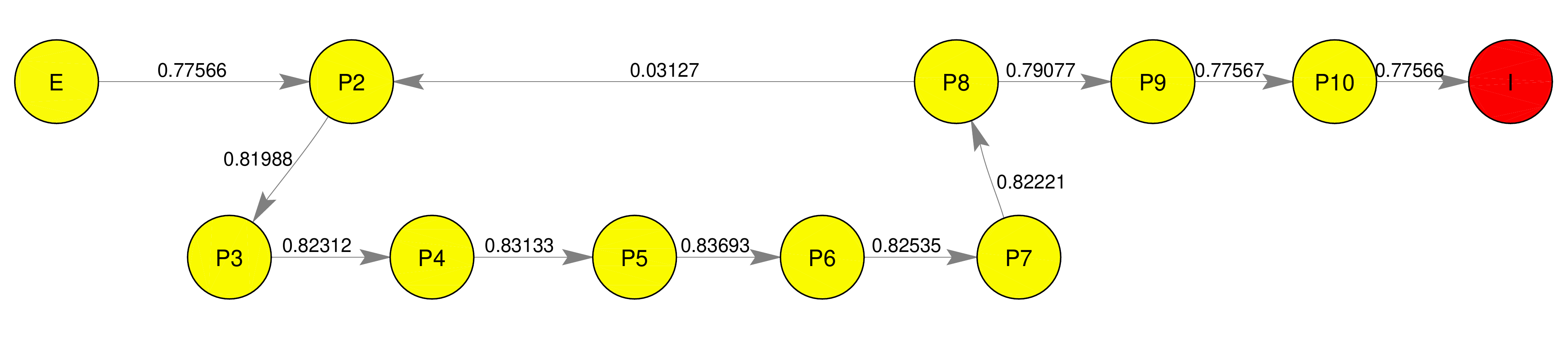}
\caption{Approximation of Ebola incubation time distribution with phase-type ($p=10$) distribution.}
\label{fig:phase}
\end{figure*}
}

For our simulations we consider a strongly connected Erdos-Renyi graph~$A$ with $N = 20$ nodes and
connection probability $0.15$. Initially, the vaccination rates $\theta_i$
are randomly chosen from a uniform distribution $\theta_i \in [0.3, 0.8]$ and the
rates of becoming susceptible $\gamma_i \in [0.2, 0.7]$. Since it is known that Ebola can only be
transmitted by people who are showing symptoms, we set $\beta_{ij}^k = 0$ for all $k \in \until{10}$.
For links that exist in the graph~$A$ we randomly set the infection rate $\beta_{ij}^{11} \in [0.1, 0.4]$.
Similarly, we assume that only infected individuals can recover, thus we set $\delta_i^k = 0$ for
all $k \in \until{10}$ and randomly set the recovery rate $\delta_i \in [0.1, 0.4]$.
Since we only have one vigilant state, there are no internal transition parameters $\mu$.

To demonstrate the effectiveness of the expansion of our model to capture the log-normal incubation
times of Ebola, we randomly set the initial conditions of being exposed to $I^1_i(0) \in [0.25, 0.75]$
and the susceptible state to $S_i(0) = 1 - I^1_i(0)$. Thus, we assume that there are initially no nodes
in the vigilant~$V$ or infected states~$I^k$ for $k \in \{2, \dots, 11 \}$. For the parameters used, 
we obtain $\lambda_\text{max}(\QQ_{xx}) = -0.1264$
as the largest real part of the eigenvalues of $\QQ_{xx}$. 
Figure~\ref{fig:trajectories}(a) shows the evolution of the maximum, minimum, and average probabilities
of being in any of the 11 infected states over time $P_i(t) = \sum_{k=1}^{11} I_i^k(t)$. Figure~\ref{fig:trajectories}(b) shows the probabilities
of being in only the truly infected (and symptomatic) state $I_i^{11}(t)$ for all nodes $i$. Here we can see
the effectiveness of our expanded model in capturing the log-normal incubation times of Ebola, seeing the
peak of infections at 12.7 days. Given enough time, we see that all infections eventually die out
as the stability condition of Theorem~\ref{th:local-stability} is satisfied. 

{
  \psfrag{onetwothreefour1}[cc][cc]{\small Maximum}%
  \psfrag{onetwothreefour2}[cc][cc]{\small Average}%
  \psfrag{onetwothreefour3}[cc][cc]{\small Minimum}%
\begin{figure}[htb]
  \centering
  \subfigure[]{\includegraphics[width=.75\linewidth]{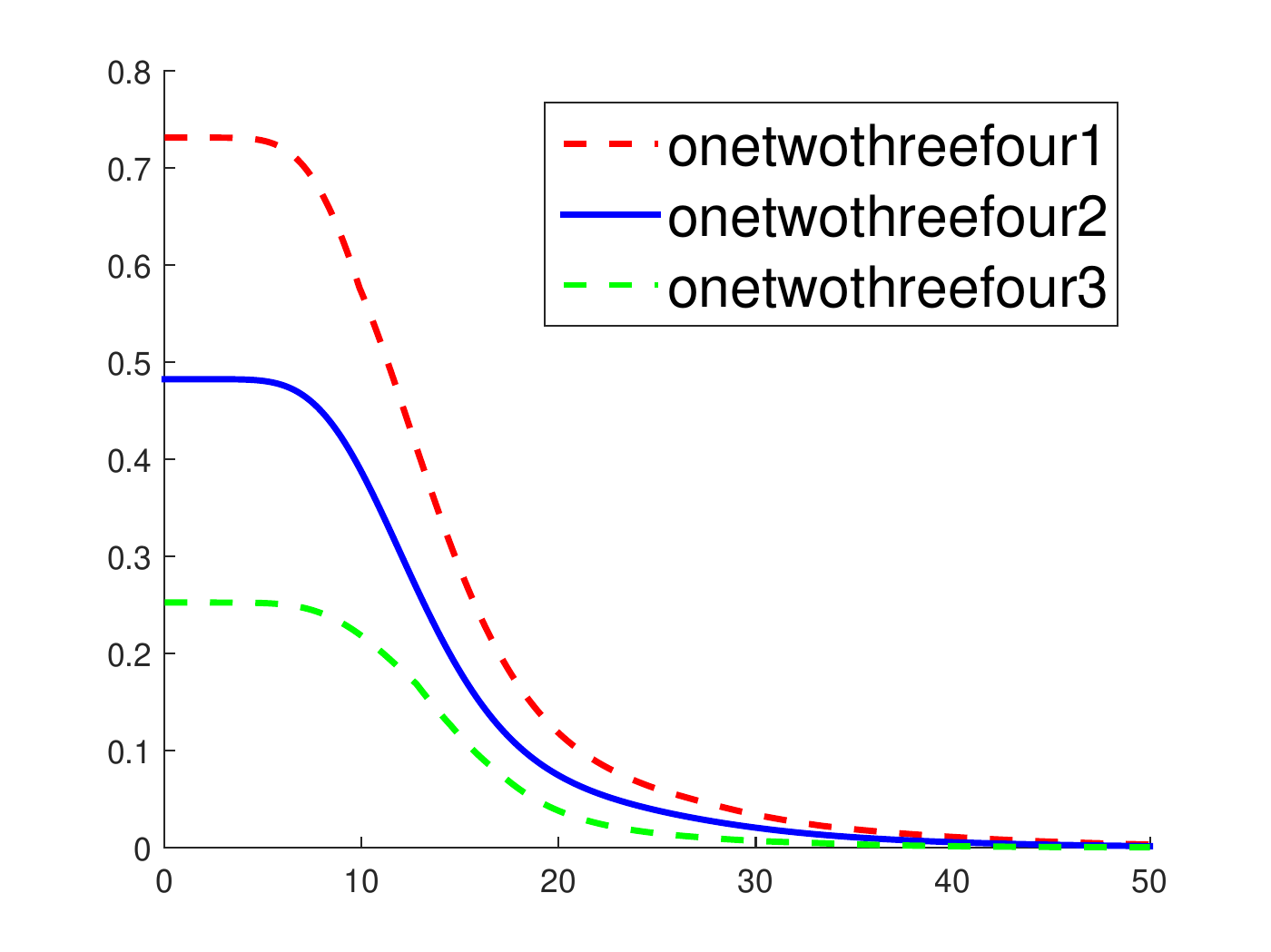}}
  \put(-100,0){\small Days} \put(-190,100){\small $P(t)$} 
  
  \subfigure[]{\includegraphics[width=.75\linewidth]{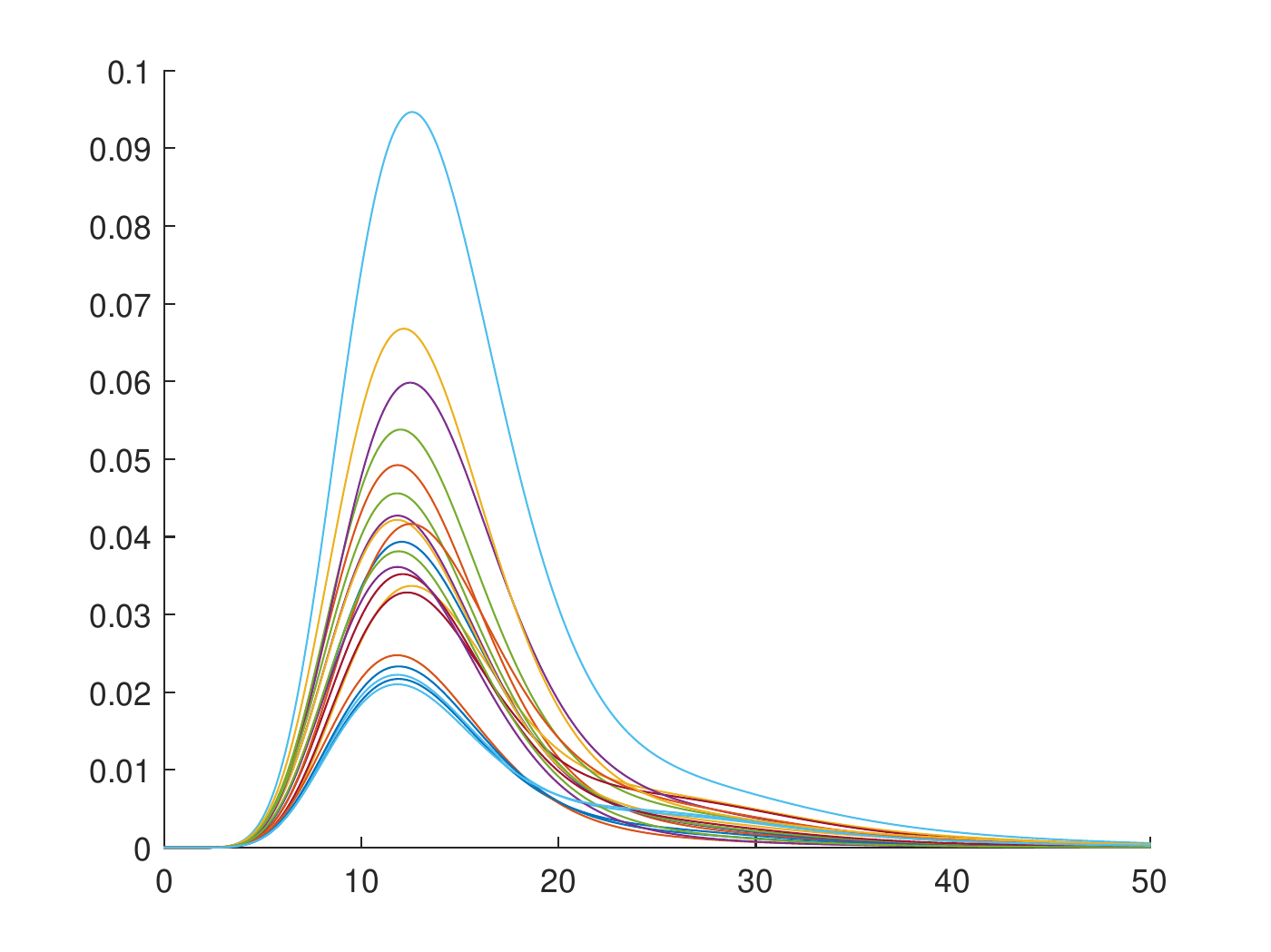}}
  \put(-100,0){\small Days} \put(-195,100){\small $I^{11}(t)$}
  \caption{Plots of (a) the minimum, maximum, and average trajectories of the probability of each node
  being in any infected state and (b) the trajectories of the probability of each node being in the
  infectious state $I^{11}$. }\label{fig:trajectories}
\end{figure}
}

In Figure~\ref{fig:steadystate} we vary the recovery rates $\delta_i$ and infection rates $\beta_{ij}^{11}$
and look at the steady-state probabilities $P(\infty)$ of each node being in any infectious state
where we approximate $P(\infty) \approx P(T)$ for large~$T$. We can see
here the sharp threshold property that occurs as $\lambda_\text{max}(\QQ_{xx})$ moves from negative to positive,
validating our stability results of Section~\ref{se:analysis}.

{
  \psfrag{onetwothreefour1}[cc][cc]{\small Maximum}%
  \psfrag{onetwothreefour2}[cc][cc]{\small Average}%
  \psfrag{onetwothreefour3}[cc][cc]{\small Minimum}%
\begin{figure}[htb]
  \centering
  {\includegraphics[width=.75\linewidth]{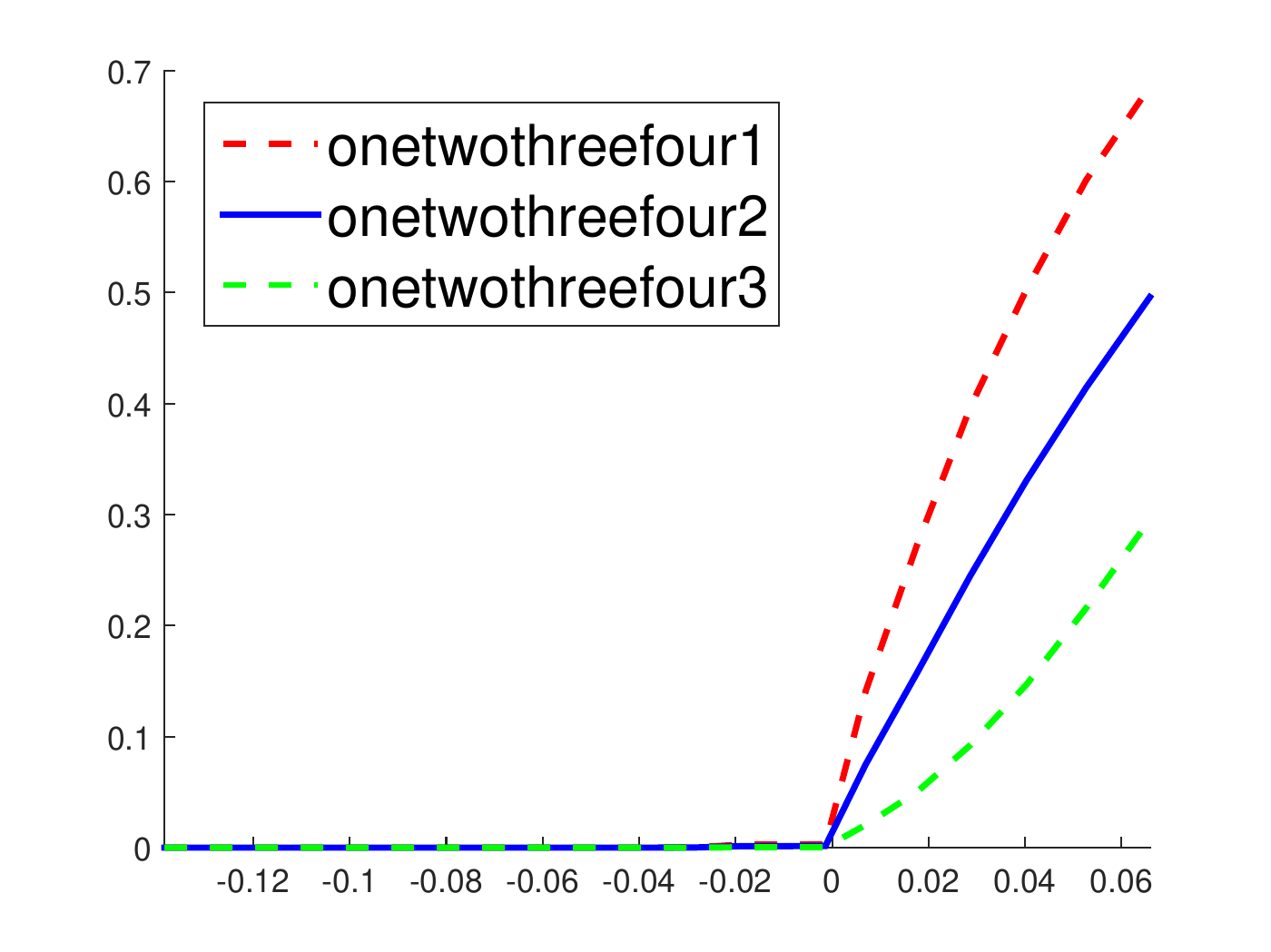}}
  \put(-100,0){\small $\lambda_\text{max}(\QQ_{xx})$} \put(-190,100){\small $P(\infty)$} 
  \caption{Plot of the minimum, maximum, and average steady-state probabilities of each node
  being in any infected state. }\label{fig:steadystate}
\end{figure}
}

\section{Conclusions}\label{se:conclusions}

In this work we have proposed a general class of stochastic epidemic models called $SI^*V^*$ on arbitrary graphs, with heterogeneous node and edge parameters, and an arbitrary number of states. We have then provided conditions
for when the disease-free equilibrium of its mean-field approximation is stable. Furthermore, we have shown
how this general class of models can be used to handle state transitions that don't follow an exponential
distribution, unlike the overwhelming majority of works in the literature. We demonstrate this modeling capability
by simulating the spreading of the Ebola virus, which is known to have a non-exponentially distributed
incubation time (i.e., time it takes to show symptoms once an individual is exposed). For future work we
are interested in studying how to control this model which can be used in a much wider range of applications
than before due to its capabilities in modeling non-Poisson spreading processes. 

\section*{Acknowledgements}
This work was supported in part by the TerraSwarm Research Center, one of six centers supported by the STARnet phase of the Focus Center Research Program (FCRP), a Semiconductor Research Corporation program sponsored by MARCO and DARPA.

%\bibliographystyle{ieeetr}
%\bibliography{alias,Main,Main-add,JC,FB,cameron,cameron-main,epidemics,masaki}

\begin{thebibliography}{10}

\bibitem{WOK-AGM:27}
W.~O. Kermack and A.~G. Mc{K}endrick, ``A contribution to the mathematical
  theory of epidemics,'' {\em Proceedings of the Royal Society A}, vol.~115,
  no.~772, pp.~700--721, 1927.

\bibitem{NTB:75}
N.~T. Bailey, {\em The Mathematical Theory of Infectious Diseases and its
  Applications}.
\newblock London: Griffin, 1975.

\bibitem{AL-JAY:76}
A.~Lajmanovich and J.~A. Yorke, ``A deterministic model for gonorrhea in a
  nonhomogeneous population,'' {\em Mathematical Biosciences}, vol.~28, no.~3,
  pp.~221--236, 1976.

\bibitem{SF-EG-CW-VAAJ:09}
S.~Funk, E.~Gilad, C.~Watkins, and V.~A.~A. Jansen, ``The spread of awareness
  and its impact on epidemic outbreaks,'' {\em Proceedings of the National
  Academy of Sciences}, vol.~16, no.~106, 2009.

\bibitem{NF:07}
N.~Furguson, ``Capturing human behaviour,'' {\em Nature}, vol.~733, no.~446,
  2007.

\bibitem{FDS-FNC-CMS:12}
F.~D. Sahneh, F.~N. Chowdhury, and C.~M. Scoglio, ``On the existence of a
  threshold for preventative behavioral responses to suppress epidemic
  spreading,'' {\em Scientific Reports}, vol.~2, no.~632, 2012.

\bibitem{CN-VMP-GJP:15-TCNS}
C.~Nowzari, V.~M. Preciado, and G.~J. Pappas, ``Optimal resource allocation in
  generalized epidemic models,'' {\em IEEE Transactions on Control of Network
  Systems}, 2015.
\newblock Submitted.

\bibitem{BAP-DC-MF-NV-CD:10}
B.~A. Prakash, D.~Chakrabarti, M.~Faloutsos, N.~Valler, and C.~Faloutsos, ``Got
  the flu (or mumps)? check the eigenvalue!,'' {\em arXiv preprint
  arXiv:1004.0060}, 2010.

\bibitem{HWH:00}
H.~W. Hethcote, ``The mathematics of infectious diseases,'' {\em SIAM Review},
  vol.~42, no.~4, pp.~599--653, 2000.

\bibitem{MMW-JL:03}
M.~M. Williamson and J.~Leveille, ``An epidemiological model of virus spread
  and cleanup,'' in {\em Virus Bulletin}, (Toronto, Canada), 2003.

\bibitem{MG-WG-DT:03}
M.~Garetto, W.~Gong, and D.~Towsley, ``Modeling malware spreading dynamics,''
  in {\em INFOCOM Joint Conference of the IEEE Computer and Communications},
  pp.~1869--1879, 2003.

\bibitem{DE-JK:10}
D.~Easley and J.~Kleinberg, {\em Networks, Crowds, and Markets: Reasoning About
  a Highly Connected World}.
\newblock Cambridge University Press, 2010.

\bibitem{MEJN:02}
M.~E.~J. Newman, ``Spread of epidemic disease on networks,'' {\em Physical
  Review E}, vol.~66, p.~016128, 2002.

\bibitem{GC-NWH-CCC-PWF-JMH:04}
G.~Chowell, N.~W. Hengartner, C.~{Castillo-Chavez}, P.~W. Fenimore, and J.~M.
  Hyman, ``The basic reproductive number of {E}bola and the effects of public
  health measures: the cases of {C}ongo and {U}ganda,'' {\em Journal of
  Theoretical Biology}, vol.~229, no.~1, pp.~119--126, 2004.

\bibitem{AK-MN-MD-MI:15}
A.~Khan, M.~Naveed, M.~{Dur-e-Ahmad}, and M.~Imran, ``Estimating the basic
  reproductive ratio for the {E}bola outbreak in {L}iberia and {S}ierra
  {L}eone,'' {\em Infectious Diseases of Poverty}, vol.~4, no.~13, 2015.

\bibitem{ME-SFD-NF:11}
M.~Eichner, S.~F. Dowell, and N.~Firese, ``Incubation period of {E}bola
  hemorrhagic virus subtype {Z}aire,'' {\em Osong Public Health and Research
  Perspectives}, vol.~2, no.~1, 2011.

\bibitem{KL-RG-TS:10}
K.~Lerman, R.~Ghosh, and T.~Surachawala, ``Social contagion: An empirical study
  of information spread on {D}igg and {T}witter follower graphs,'' in {\em
  Proceedings of the Fourth International AAAI Conference on Weblogs and Social
  Media}, (Washington, DC), pp.~90--97, 2010.

\bibitem{PVM-NB-CD:11}
P.~V. Mieghem, N.~Blenn, and C.~Doerr, ``Lognormal distribution in the {D}igg
  online social network,'' {\em The European Physical Journal B}, vol.~83,
  no.~2, pp.~251--261, 2011.

\bibitem{CD-NB-PVM:13}
C.~Doerr, N.~Blenn, and P.~V. Mieghem, ``Lognormal infection times of online
  information spread,'' {\em PLoS ONE}, vol.~8, p.~e64349, Nov. 2013.

\bibitem{PVM-RVDB:13}
P.~V. Mieghem and R.~{van de Bovenkamp}, ``Non-{M}arkovian infection spread
  dramatically alters the {S}usceptible-{I}nfected-{S}usceptible epidemic
  threshold in networks,'' {\em Physical Review Letters}, vol.~110, no.~10,
  p.~108701, 2013.

\bibitem{HJ-JIP-KK-JK:14}
H.~Jo, J.~I. Perotti, K.~Kaski, and J.~Kertesz, ``Analytically solvable model
  of spreading dynamics with non-{P}oissonian processes,'' {\em Physical Review
  X}, vol.~4, no.~1, p.~011041, 2014.

\bibitem{EC-RVDB-PVM:13}
E.~Cator, R.~{van de Bovenkamp}, and P.~V. Mieghem,
  ``Susecptible-{I}nfected-{S}usceptible epidemics on networks with general
  infection and cure times,'' {\em Physical Review E}, vol.~87, p.~062816,
  2013.

\bibitem{AR:11}
A.~Rantzer, ``Distributed control of positive systems,'' in {\em {IEEE} Conf.\
  on Decision and Control}, (Orlando, FL), pp.~6608--6611, 2011.

\bibitem{LF-SR:00}
L.~Farina and S.~Rinaldi, {\em Positive Linear Systems}.
\newblock Wiley-Interscience, 2000.

\bibitem{PVM-JO-RK:09}
P.~V. Miegham, J.~Omic, and R.~Kooij, ``Virus spread in networks,'' {\em
  IEEE/ACM Transactions on Networking}, vol.~17, no.~1, pp.~1--14, 2009.

\bibitem{MJK-PR:07}
M.~J. Keeling and P.~Rohani, {\em Modeling Infectious Diseases in Humans and
  Animals}.
\newblock Princeton University Press, 2007.

\bibitem{PVM:09}
P.~V. Mieghem, {\em Performance Analysis of Communications Networks and
  Systems}.
\newblock Cambridge, UK: Cambridge University Press, 2009.

\bibitem{HKK:02}
H.~K. Khalil, {\em Nonlinear Systems}.
\newblock Prentice Hall, 3~ed., 2002.

\bibitem{SA-ON-MO:96}
S.~Asmussen, O.~Nerman, and M.~Olsson, ``Fitting phase-type distributions via
  the {EM} algorithm,'' {\em Scandinavian Journal of Statistic}, vol.~23,
  no.~4, pp.~419--441, 1996.

\bibitem{DRC:55}
D.~R. Cox, ``A use of complex probabilities in the theory of stochastic
  processes,'' {\em Mathematical Proceedings of the Cambridge Philosophical
  Society}, vol.~51, no.~2, pp.~313--319, 1955.

\end{thebibliography}

\end{document}